\documentclass[12pt,a4paper]{article}
\usepackage{amsthm,amssymb,amsmath,mathrsfs,graphicx,enumerate}
\usepackage[pagebackref=false,colorlinks,linkcolor=black,citecolor=magenta]{hyperref}
\usepackage{xcolor}
\allowdisplaybreaks[1]

\newcommand{\wt}{\widetilde}

\newcommand{\soc}{{\rm Soc}}

\newcommand{\Aut}{{\rm Aut}}
\newcommand{\Sym}{{\rm Sym}}

\newcommand{\Syl}{{\rm Syl}}

\newtheorem{theorem}{\bf Theorem}
\newtheorem*{theoremA}{\bf Theorem A}
\newtheorem*{theoremB}{\bf Theorem B}
\newtheorem*{theoremC}{\bf Theorem C}
\newtheorem*{theoremD}{\bf Theorem}
\newtheorem{lemma}{\bf Lemma}[section]

\newtheorem{corollary}[lemma]{\bf Corollary}
\newtheorem{proposition}[lemma]{\bf Proposition}

\date{}

\title{On finite totally $2$-closed groups}

\author{Alireza Abdollahi\footnote{E-mail: a.abdollahi@math.ui.ac.ir} $^{a}$,
 Majid Arezoomand\footnote{E-mail: arezoomand@lar.ac.ir (Corresponding author)} $^{b}$, and Gareth Tracey\footnote{E-mail: g.tracey@bham.ac.uk}\footnote{The third author would like to thank the Engineering and Physical Sciences Research Council for their support via the grant EP/T017619/1.} $^{c}$\\ {\small\em $^a$ Department of Pure Mathematics, Faculty of Mathematics and Statistics,} \\
{\small\em University of Isfahan, Isfahan 81746-73441, Iran}\\
 {\small\em $^b$ University of Larestan, Larestan 74317-16137, Iran}\\
 {\small\em $^c$ School of Mathematics, University of Birmingham,}\\ {\small\em Edgbaston, Birmingham, B15 2TT, United Kingdom} }

\begin{document}
\maketitle

\begin{abstract}
An abstract group $G$ is called totally $2$-closed if $H=H^{(2),\Omega}$ for any set $\Omega$ with $G\cong H\leq\Sym(\Omega)$,
where $H^{(2),\Omega}$ is the largest subgroup of $\Sym(\Omega)$ whose orbits on $\Omega\times\Omega$ are the same orbits
of $H$. In this paper, we classify the finite soluble totally $2$-closed groups. We also prove that the Fitting subgroup of a totally $2$-closed group is a totally $2$-closed group. Finally, we prove that a finite insoluble totally $2$-closed group $G$ of minimal order with non-trivial Fitting subgroup has shape $Z\cdot X$, with $Z=Z(G)$ cyclic, and $X$ is a finite group with a unique minimal normal subgroup, which is nonabelian.


\end{abstract}

\section{Introduction and results} 
Let $\Omega$ be a set and $G$ be a group with $G\leq\Sym(\Omega)$. Then $G$ acts naturally on
$\Omega\times\Omega$ by
$(\alpha_1,\alpha_2)^g=(\alpha_1^g,\alpha_2^g)$, where $g\in G$ and $\alpha_1,\alpha_2\in\Omega$.
The {\it $2$-closure} of $G$ on $\Omega$, denoted by $G^{(2),\Omega}$,
 is defined to be the subgroup of $\Sym(\Omega)$ leaving each orbit of $G$ on
 $\Omega\times\Omega$ fixed. By \cite[Theorem 5.6]{Wielandt}
\[G^{(2),\Omega}=\{\theta\in\Sym(\Omega)\mid \forall \alpha,\beta\in\Omega, \exists g\in G, \alpha^\theta=\alpha^g,\beta^\theta=\beta^g\}.\]
 Furthermore, $G^{(2),\Omega}$ contains $G$ and is the  largest subgroup
of $\Sym(\Omega)$ whose orbits on $\Omega\times\Omega$ are the same orbits of $G$ \cite[Definition 5.3 and Theorem 5.4]{Wielandt}. Indeed $G^{(2),\Omega}$ is the automorphism group of the set of all $2$-ary relations invariant with respect to the group $G\leq\Sym(\Omega)$. Also $G$ is $2$-closed on $\Omega$, i.e $G=G^{(2),\Omega}$ if and only if there exists a complete colored digraph $\Gamma$ with vertex set $\Omega$  such that $\Aut(\Gamma)=G$. The partition of the set of arcs of this graph induced by the coloring forms a set of relations which generates the Krasner algebra \cite[p. 15]{Faradzev}.

In 1969,
Wielandt initiated the study of $2$-closures of permutation groups to present a
unified treatment of finite
and infinite permutation groups, based on  invariant relations and invariant functions \cite{Wielandt}. After Wielandt's pioneering work, there was some progress on the subject achieved mostly in the case of primitive groups \cite{Liebeck0, Liebeck1, Onan, Praeger1, Praeger2, Xu} and the $2$-closure was used as a tool in studying the graph isomorphism problem \cite{Ponomarenko1,Ponomarenko2,PV}; the isomorphism problem for Schurian coherent configurations \cite{Faradzev,VC}; and in the study of automorphisms of vertex transitive graphs \cite{Dobson,xu2011,xu2015}. The latter of these led to the formulation of the Polycirculant conjecture \cite{Cam-Giu}, which remains open, and has garnered much recent attention \cite{AAS}.

Due to its widespread motivation, the $2$-closure has been studied extensively. In particular, an interesting open question asks how far $G^{(2),\Omega}$ can be from G. This question was answered in \cite{Liebeck1} in the case where $G$ is a primitive almost simple permutation group, but remains open in general. In this paper, we study those finite groups which have the property that $G^{(2),\Omega}=G$ for all faithful permutation representations $G\leq\Sym(\Omega)$. Motivated by the study of some combinatorial invariants in lattice theory, Monks proves in \cite{Monks} that a finite cyclic group satisfies this property. Derek Holt in his answer to a question  \cite{holt} proposed
by the second author in Mathoverflow  introduced a class of abstract groups
called totally $2$-closed groups consisting of all groups which are $2$-closed in all of their
faithful permutation representations. Significant progress on the question of classifying the finite totally $2$-closed groups was achieved in \cite{AIPT}. There, the finite totally $2$-closed groups with trivial Fitting subgroup were classified: 
\begin{theoremD}\cite[Theorem 1.2 and Corollary 1.3]{AIPT}
		Let $G$ be a non-trivial finite group with trivial Fitting subgroup. Then $G$ is totally $2$-closed if and only if each of the following holds:
		\begin{enumerate}[\upshape(1)]
			\item $G=T_1\times\hdots\times T_r$, where the $T_i$ are nonabelian finite simple groups and $r\le 5$; 
			\item $T_i\not\cong T_j$ for each $i\neq j$; and 
			\item One of the following holds:
			\begin{enumerate}[\upshape(i)]
				\item $T_i\in\{\mathrm{J}_1, \mathrm{J}_3, \mathrm{J}_4, \mathrm{Th}, \mathrm{Ly}\}$ for each $i\le r$; or
				\item $T_i\in \{\mathrm{J}_1, \mathrm{J}_3, \mathrm{J}_4,\mathrm{Ly}, \mathbb{M}\}$ for each $i\le r$.
			\end{enumerate}
		\end{enumerate}
\end{theoremD}
In particular, there are precisely $47$ totally $2$-closed finite groups with trivial Fitting subgroup.
In this paper, we classify the finite soluble totally $2$-closed groups. Our main result reads as follows:
\begin{theoremA}
Let $G$ be a finite soluble group. Then $G$ is totally $2$-closed if and only if it is cyclic or a direct product of a cyclic
 group of odd order with a generalized quaternion group.
\end{theoremA}

This extends the main theorem in \cite{arezoomand}, which classifies the finite nilpotent totally $2$-closed groups. As the reader can see, all the groups in Theorem A are nilpotent, so we deduce that there are no non-nilpotent soluble totally $2$-closed groups. 

As a by-product of our methods, we obtain a structural theorem for general finite totally $2$-closed groups: We show that the Fitting subgroup of such a group is also totally $2$-closed:
\begin{theoremB}
 Let $G$ be a totally $2$-closed finite group. Then $F(G)$ is totally $2$-closed. In particular, 
 $F(G)$ is cyclic or a direct product of a generalized quaternion group
 with a cyclic group of odd order.
\end{theoremB}

Finally, with these results, and the results in \cite{AIPT} in mind, we were interested in obtaining structural information on a finite insoluble totally $2$-closed group with non-trivial Fitting subgroup, of minimal order (such a group may not exist, of course). Our final theorem reads as follows:
\begin{theoremC}
A finite insoluble totally $2$-closed group $G$ of minimal order with non-trivial Fitting subgroup has shape $Z.X$, where $Z=Z(G)$ is cyclic, and $X$ has a unique minimal normal subgroup, which is nonabelian.
\end{theoremC}

\subsection{Preliminary results and notations}
 In this section we collect some basic and elementary results and notations we need later. Our notations are standard and are mainly taken from \cite{Dixon}, but for
the reader's convenience  we recall some of them as  follows:
\begin{itemize}
\item[] $\Sym(\Omega)$: The symmetric group on the set $\Omega$.
\item[] $\alpha^g$: The action of $g$ on $\alpha$.
\item[] $G_\alpha$: The point stabilizer of $\alpha$ in $G$.
\item[] $\alpha^G$: The orbit of $\alpha$ under $G$.
\item[] $Z(G)$: The center of $G$.
\item[] $F(G)$: The fitting subgroup of $G$.
\item[] $C_G(H)$: The centralizer of the  subgroup $H$ of a group $G$.
\item[] $O_p(G)$: The intersection of all Sylow $p$-subgroups of $G$.
\item[] $\soc(G)$: The socle of a group $G$.
\item[] $H_G$: The core of the  subgroup $H$ of $G$, that is the intersection of all $G$-conjugates of $H$.
\end{itemize}
 
 \begin{lemma}{\rm (\cite[Lemma 2.1]{arezoomand})}\label{Lemma2.1arezoomand}
Let $G\leq\Sym(\Omega)$, $A,B\leq G$ and $[A,B]=1$. Then $[A^{(2),\Omega},B^{(2),\Omega}]=1$. In particular, if $G$
is abelian then $G^{(2),\Omega}$ is also abelian. Furthermore,
if $H\leq G$ then $(C_G(H))^{(2),\Omega}\leq C_{G^{(2),\Omega}}(H^{(2),\Omega})$.
\end{lemma}

\begin{lemma}{\rm (\cite[Lemma 2.2]{arezoomand})}\label{Lemma2.2arezoomand}
Let $G\leq\Sym(\Omega)$ and $H\leq\Sym(\Gamma)$ be permutation isomorphic. Then
$G^{(2),\Omega}\leq\Sym(\Omega)$ and $H^{(2),\Gamma}\leq\Sym(\Gamma)$ are permutation isomorphic.
\end{lemma}

\begin{lemma}{\rm (\cite[Lemma 2.3]{arezoomand})}\label{Lemma2.3arezoomand}
Let $G\leq\Sym(\Omega)$ and $x\in\Sym(\Omega)$. Then $(x^{-1}Gx)^{(2),\Omega}=x^{-1}G^{(2),\Omega}x$. In particular, $N_{\Sym(\Omega)}(G)\leq N_{\Sym(\Omega)}(G^{(2),\Omega})$.
\end{lemma}

\begin{lemma}{\rm (\cite[Lemma 2.9]{arezoomand})}\label{Lemma2.9arezoomand} Let $G_i\leq\Sym(\Omega_i)$, $\Omega$ be disjoint union of $\Omega_i$'s, $i=1,\ldots,n$ and
$G=G_1\times\cdots\times G_n$. Then the natural action  of $G$ on $\Omega$ is faithful. Furthermore, if $G^{(2),\Omega}=G$
then for each $i=1,\ldots,n$,
$G_i^{(2),\Omega_i}=G_i$. In particular, if $G$ is a totally $2$-closed group then $G_i$ is a totally $2$-closed group, $i=1,\ldots,n$.
\end{lemma}

\begin{lemma}{\rm (\cite[Lemma 4.5]{arezoomand})}\label{Lemma4.5arezoomand}
 Let $G=H\times K\leq\Sym(\Omega)$ be transitive and $\Omega=\alpha^G$. If $(|H|,|K|)=1$, then
the action of $G$ on $\Omega$ is
equivalent to the action of $G$ on $\Omega_1\times\Omega_2$, where $\Omega_1=\alpha^H$, $\Omega_2=\alpha^K$
and $G$ acts on $\Omega_1\times\Omega_2$ by the rule $(\alpha^h,\alpha^k)^{g}=(\alpha^{hh_1},\alpha^{kk_1})$, where $g=h_1k_1$.
\end{lemma}

\begin{theorem}{\rm(\cite[Theorem 1]{arezoomand})}\label{Theorem1arezoomand}
The center of every finite totally $2$-closed group is cyclic.
\end{theorem}

\begin{theorem}{\rm(\cite[Theorem 2]{arezoomand})}\label{Theorem2arezoomand} 
A finite nilpotent  group is totally $2$-closed if and only if it is cyclic or a direct product of a generalized quaternion group with a cyclic group of odd order.
\end{theorem}

\begin{theorem}{\rm(\cite[Theorem 5.1]{Cam-Giu})}\label{Theorem5.1Cam-Giu}
Let $G_1$ and $G_2$ be transitive permutation groups on sets $\Omega_1$ and $\Omega_2$, respectively. Then in their action on $\Omega:=\Omega_1\times\Omega_2$, we have 
\[(G_1\times G_2)^{(2),\Omega}=G_1^{(2),\Omega_1}\times G_2^{(2),\Omega_2},~~(G_1\wr G_2)^{(2),\Omega}=G_1^{(2),\Omega_1}\wr G_2^{(2),\Omega_2}.\]
Hence the following are equivalent:
\begin{itemize}
\item[(a)] $G_1$ and $G_2$ are $2$-closed on $\Omega_1$ and $\Omega_2$, respectively.
\item[(b)] $G_1\times G_2$ is $2$-closed on $\Omega$.
\item[(c)] $G_1\wr G_2$ is $2$-closed on $\Omega$.
\end{itemize}
\end{theorem}

\begin{theorem}{\rm(\cite[Dissection Theorem 6.5]{Wielandt})}\label{DissectionTheorem6.5Wielandt}
Let $G$ acts on a set $\Omega$, and suppose that $\Omega=\Delta\cup\Gamma$ (disjoint union), where $\Gamma$ is $G$-invariant. Then
the following are equivalent:
\begin{itemize}
\item[(1)] $G^\Gamma\times G^\Delta\leq G^{(2),\Omega}$.
\item[(2)] $G=G_\gamma G_\delta$ for all $\gamma\in\Gamma$ and $\delta\in\Delta$.
\item[(3)] $G_\delta$ is transitive on $\gamma^G$ for all $\gamma\in\Gamma$ and $\delta\in\Delta$.
\end{itemize}
\end{theorem}

\begin{theorem}{\rm (Universal embedding theorem \cite[Theorem 2.6 A]{Dixon})}\label{universal} Let $G$ be an arbitrary group
with a normal subgroup $N$ and put $K:=G/N$. Let $\psi: G\rightarrow K$ be a homomorphism of $G$ onto $K$ with kernel $N$.
Let $T:=\{t_u\mid u\in K\}$ be a set of right coset representatives of $N$ in $G$ such that $\psi(t_u)=u$ for each $u\in K$. Let
$x\in G$ and $f_x:K\rightarrow N$ be the map with $f_x(u)=t_uxt_{u\psi(x)}^{-1}$ for all $u\in K$. Then $\varphi(x):=(f_x,\psi(x))$
defines an embedding $\varphi$ of $G$ into $N\wr K$. Furthermore, if $N$ acts faithfully on a set $\Delta$ then $G$ acts
faithfully on $\Delta\times K$ by the rule $(\delta,k)^{x}=(\delta^{f_x(k)},k\psi(x))$.
\end{theorem}

 \section{The proof of Theorem A} 
In this section, we will prove that every finite soluble totally $2$-closed group is nilpotent and so it is a cyclic group of a direct product of 
a cyclic group of odd order with a generalized quaternion group. Using the following lemma, one can eliminate lots of candidates of totally $2$-closed groups:

\begin{lemma}\label{sd}
Let $G=HK$ be a totally $2$-closed group, where $H,K$ are proper subgroups of $G$.  If $H_G\cap K_G=1$, then $G=H_G\times K_G$ 
and both $H_G$ and $K_G$ are totally $2$-closed. In particular, if $H\cap K=1$, then $G=H\times K$ and  both $H$ and $K$ are totally $2$-closed
\end{lemma}
\begin{proof}
Let $\Gamma$ and $\Delta$ be the set of right cosets of $H$ and $K$, respectively and $\Omega=\Gamma\cup\Delta$.
Then for all $x,y\in G$, we have $G=H^xK^y$ \cite[Problem 1A.4]{Isaacs}. Consider the actions of $G$ on $\Gamma$, $\Delta$ and
$\Omega$ by right multiplication. Since $H_G\cap K_G=1$, $G$ acts on $\Omega$ faithfully. Also,
 since $G_{Hx}=H^x$ and $G_{Ky}=K^y$ for all $x,y\in G$, the Wielandt's Dissection 
Theorem \cite[Dissection Theorem 5.6]{Wielandt} implies that $G/H_G\times G/K_G=G^\Gamma\times G^\Delta\leq G^{(2),\Omega}$.
Now since $G$ is
totally $2$-closed, $G^{(2),\Omega}=G$ and so  $G=H_GK_G$, which implies  that $G=H_G\times K_G$. 
Now \ref{Lemma2.9arezoomand} implies that $H_G$ and $K_G$ are both totally $2$-closed groups.
\end{proof}

In the following corollary, we determine the structure of normal Sylow subgroups of totally $2$-closed groups.
\begin{corollary}
Let $G$ be a finite totally $2$-closed group. Then every normal Sylow subgroup of $G$ is cyclic or a generalized quaternion group.
\end{corollary}
\begin{proof} Let $P$ be a normal Sylow subgroup of $G$. Then $G=P\rtimes H$ \cite[3.8 Theorem]{Isaacs}, for some subgroup $H$ of $G$.  Hence,
by Lemma \ref{sd} and Theorem \ref{Theorem2arezoomand}, $P$ is cyclic or a generalized quaternion group.
\end{proof}

\begin{proof}[Proof of Theorem A.]
One direction is clear by Theorem \ref{Theorem2arezoomand}. Conversely, suppose that $G$ be a finite soluble totally $2$-closed group.
Let $p$ be a prime divisor of $|G|$ and $P\in\Syl_p(G)$. Since $G$ is soluble, it has a Hall $p'$-subgroup $H$. Hence
$G=HP$, where $H\cap P=1$. Since $G$ is totally $2$-closed, Lemma \ref{sd} implies that $P\unlhd G$. This means that $G$ is a nilpotent
group. Now Theorem \ref{Theorem2arezoomand} implies that $G$ is cyclic or a direct product of a cyclic group of odd order with a generalized
quaternion group.\end{proof}

\begin{corollary} Let $G$ be a finite group of even order which has a cyclic Sylow $2$-subgroup. Then $G$ is totally
$2$-closed if and only if $G$ is cyclic. In particular, if $G$ is a group of order $2n$, where $n$ is odd, then
$G$ is totally $2$-closed if and only if $G$ is cyclic.
\end{corollary}
\begin{proof} Since the Sylow $2$-subgroup of $G$ is cyclic, it has a normal $2$-complement \cite[5.14 Corollary]{Isaacs}. Hence $G=H\rtimes P$, where $P\in\Syl_p(G)$. 
If $G$ is totally
$2$-closed, then Lemma \ref{sd}, implies that either $H=1$ or $G=H\times P$, where $H$ and $P$ are both totally $2$-closed. In the first
case there is nothing to prove and in the latter case, $G$ is cyclic by Theorem A and the Feit-Thompson odd-order Theorem.
\end{proof}

 \section{The proof of Theorem B}
By the Exercise 5.28 of Wielandt's book  \cite[Exercise 5.28]{Wielandt}, a finite group $G\leq\Sym(\Omega)$ is a $p$-group  if and only if 
$G^{(2),\Omega}$ is a $p$-group. First, for the completeness, we prove this exercise and generalize it to nilpotent groups.
 \begin{lemma}\label{ptrans}
Let $G\leq\Sym(\Omega)$ be a transitive $p$-group, $|\Omega|<\infty$. Then $G^{(2),\Omega}$ is a $p$-group.
\end{lemma}
\begin{proof}
Since $G$ is a transitive $p$-group, $|\Omega|=p^k$, for some integer $k\geq 1$. Also there
exists $P\in\Syl_p(\Sym(\Omega))$ such that $G\leq P$. Let 
$\Delta=\{\alpha_1,\ldots,\alpha_p\}$ be a subset of $\Omega$ of size $p$ and
$C=\langle (\alpha_1,\ldots,\alpha_p)\rangle$. Define recursively: $P_1=C$ acting on $\Delta$; and $P_m=P_{m-1}\wr_\Delta C$ acting
on $\Delta^m$ for $m\geq 2$. Then $P_m$ acts faithfully on $\Delta^m$. Now $P\leq\Sym(\Omega)$ is permutation isomorphic
to $x^{-1}P_{k}x\leq\Sym(\Delta^k)$ for some $x\in\Sym(\Delta^k)$, for more details see \cite[Example 2.6.1]{Dixon}, and so it
is permutation isomorphic to $P_k\leq\Sym(\Delta^k)$. Now Theorem \ref{Theorem5.1Cam-Giu} implies that $P_k^{(2)}$ is $2$-closed
on $\Delta^k$. Hence $P$ is $2$-closed on $\Omega$, by Lemma \ref{Lemma2.2arezoomand}. Thus $G^{(2),\Omega}\leq P$ is a $p$-group.
\end{proof}

\begin{corollary}\label{p-group} Let $G\leq\Sym(\Omega)$, $|\Omega|<\infty$ and $p$ be a prime. Then $G$ is a $p$-group if and only if $G^{(2),\Omega}$ is
a $p$-group. 
\end{corollary}
\begin{proof} Let $G$ have $m$ orbits $\Omega_1,\ldots,\Omega_m$ on $\Omega$. Since the orbits of $G^{(2)}$ on $\Omega$ are the same orbits of $G$, we have
$G^{(2),\Omega}\leq G^{(2),\Omega_1}\times\cdots\times G^{(2),\Omega_m}$. On the other hand, by \cite[Exercise 5.25]{Wielandt}, 
for each $i$ we have
$G^{(2),\Omega_i}\leq (G^{\Omega_i})^{(2),\Omega_i}$, $i=1,\ldots,m$.  Since $G^{\Omega_i}$ is a transitive permutation $p$-group,
Lemma \ref{ptrans} implies that for each $i$, $(G^{\Omega_i})^{(2),\Omega_i}$ is a $p$-group. Hence $G^{(2),\Omega}$ is a $p$-group. The
converse direction is clear, and the proof is complete. 
\end{proof}

\begin{corollary}\label{O_p}
Let $P$ be a Sylow $p$-subgroup of a finite group $G\leq\Sym(\Omega)$. If $G^{(2),\Omega}=G$ then $P^{(2),\Omega}=P$ and 
$(O_p(G))^{(2),\Omega}=O_p(G)$.
\end{corollary}

\begin{corollary}\label{nilp2}
Let $G\leq\Sym(\Omega)$ and $|\Omega|<\infty$. Then $G$ is nilpotent if and only if $G^{(2),\Omega}$ is nilpotent.
\end{corollary}
\begin{proof} Let $|G|=p_1^{n_1}\ldots p_r^{n_r}$, where $p_1<p_2<\cdots<p_r$ be primes and $n_1,\ldots,n_r\geq 1$ be integers. 
Let $P_i$ be the Sylow $p_i$-subgroup of $G$. Then $G=P_1\times\cdots\times P_r$.
First suppose that $G$ is transitive on $\Omega$ and $\alpha^G=\Omega$. Then Lemma \ref{Lemma4.5arezoomand} implies that 
the action of $G$ on $\Omega$ is  equivalent to the pointwise action of $G$ on $\Omega_1\times\cdots\times\Omega_r$, where 
$\Omega_i=\alpha^{P_i}$. Let $\Delta=\Omega_1\times\cdots\times\Omega_r$.
Since $P_i$ acts transitively on $\Omega_i$, Theorem \ref{Theorem5.1Cam-Giu} implies that, in the pointwise action on 
$\Delta$, $G^{(2)}=P_1^{(2)}\times\cdots\times P_r^{(2)}$. On the other hand, by Corollary \ref{p-group}, $P_i^{(2)}$ is
a $p_i$-group and so $P_i^{(2)}$ is a Sylow $p_i$-subgroup of $G^{(2)}$, which means that $G^{(2)}$ is a nilpotent group.

Now let $G$ have $m$ orbits $\Omega_1,\ldots,\Omega_m$ on $\Omega$. Then by a similar argument in the proof of Corollary \ref{p-group},
$G^{(2),\Omega}\leq (G^{\Omega_1})^{(2),\Omega_1}\times\cdots\times (G^{\Omega_m})^{(2),\Omega_m}$. Since for each $i$, 
$G^{\Omega_i}$ is a 
transitive nilpotent permutation group, the above argument follows that $G^{(2),\Omega}$ is nilpotent.\end{proof}

Next, we discuss the Universal Embedding Theorem. Let $G$ be a finite group, and $N$ a normal subgroup of $G$. Suppose that $N$ acts faithfully on a set $\Delta$. Let $\Gamma:=G/N$, and let $T$ be a right transversal for $N$ in $G$. Also, for $g=Ny\in \Gamma$, let $t_g$ be the unique element of $T$ such that $Ny=Nt_g$. Then $G$ acts faithfully on the finite set $\Omega:=\Delta\times \Gamma$ via the rule $(\delta,g)^x:=(\delta^{t_gxt_{g\psi(x)}^{-1}},g\psi(x))$, where $\psi:G\rightarrow \Gamma$ is a homomorphism of $G$ onto $K$ with kernel $N$. Let $\chi$ be the permutation character of $N$ acting on $\Delta$. Then the permutation character of $G$ acting on $\Omega$ is the induced character $\chi\uparrow^G_N$. Whence, Mackey's Theorem \cite[Proposition 6.20]{Mack} implies that the permutation character of $N$ on $\Omega$ is $\chi\uparrow^G_N\downarrow_N=\sum_{i=1}^m\chi_{N^{x_i}\cap N}\uparrow^N$, where $\{x_1,\hdots,x_m\}$ is a set of representatives for the $(N,N)$ double cosets in $G$. Since $N$ is normal in $G$, it follows that $N^{x_i}\cap N=N$, and in fact that $\{x_1,\hdots,x_m\}$ is a set of representatives for the right cosets of $N$ in $G$. Thus, we conclude that $\chi\uparrow^G_N\downarrow_N= m\chi$. That is, the permutation action of $N$ on $\Delta$ is permutation isomorphic to the natural action of $N$ on a disjoint union of $[G:N]$ copies of $\Delta$. In fact, this permutation isomorphism can be viewed as follows: the orbits of $N$ in its action on $\Omega$ are the sets $\Delta_g:=\{(\delta,g)\text{ : }\delta\in\Delta\}$, for $g\in \Gamma$. The permutation isomorphism $\theta_g:(N,\Delta_g)\rightarrow (N,\Delta)$ is given by $n\rightarrow t_gnt_g^{-1}$, $(\delta,g)\rightarrow \delta$.

Next, let $\mathcal{P}$ be a group theoretic property which is closed under normal extension. That is, if $H$ and $K$ are normal $\mathcal{P}$-subgroups of a finite group $G$, then $HK$ is a (necessarily) normal $\mathcal{P}$-subgroup of $G$. Thus, we can define the largest normal $\mathcal{P}$-subgroup of any finite group $G$: we denote this subgroup by $O_{\mathcal{P}}(G)$. Examples of group theoretic properties which are closed under normal extension include nilpotency and solubility.

\begin{lemma}\label{Second}
 Let $\mathcal{P}$ be a group theoretic property which is closed under normal extension with the property that whenever $\Omega$ is a finite set with $X\le \Sym(\Omega)$ a $\mathcal{P}$-subgroup, we have that $X^{(2),\Omega}$ has property $\mathcal{P}$. Suppose that $G\le \Sym(\Omega)$ is $2$-closed. Then $O_{\mathcal{P}}(G)$ is $2$-closed.
 \end{lemma}
\begin{proof}
 By Lemma \ref{Lemma2.3arezoomand}, the group $O_{\mathcal{P}}(G)^{(2),\Omega}$ is normal in $G=G^{(2),\Omega}$. The hypothesis on $\mathcal{P}$ then guarantees that $O_{\mathcal{P}}(G)^{(2),\Omega}$ is $\mathcal{P}$, and hence that $O_{\mathcal{P}}(G)^{(2),\Omega}\le O_{\mathcal{P}}(G^{(2),\Omega})=O_{\mathcal{P}}(G)$. The result follows.
\end{proof}

\begin{proposition}\label{FTheorem}
 Let $G$ be a finite totally $2$-closed group. Then the Fitting subgroup $F(G)$ of $G$ is totally $2$-closed.
 \end{proposition}
\begin{proof} 
Let $F:=F(G)$ and let $\Delta$ be a set on which $F$ acts faithfully. Let $\Gamma:=G/F$ and $\Omega:=\Delta\times\Gamma$. Then $G$ acts faithfully on $\Omega$ by the Universal Embedding Theorem. With this embedding, we have that $F(G)^{(2),\Omega}=F(G)$, by Lemma \ref{Second} and Corollary \ref{nilp2}.

Now, for each $g\in \Gamma$, let $F_g:=F^{(2),\Delta_g}\le\Sym(\Delta_g)$, and let $\mu_g:F_1\rightarrow F_g$ be a permutation isomorphism (see the paragraph above for an explanation of this notation). Then $F_1$ acts faithfully on $\Omega$ by the rule $(\delta,g)^z=(\delta^{\mu_g(z)},g)$, for $\delta\in\Delta$, $g\in \Gamma$, and $z\in F_1$. Furthermore, the natural copy $\wt{F}$ of $F$ in $F_1$ acts faithfully on $\Omega$ via restriction, and by the paragraph above this action is permutation isomorphic to the action of $F$ on $\Omega$ coming from the Universal Embedding Theorem. 
Denote by  $\wt{F}'$ and $F_1'$ the images of $\wt{F}$ and $F_1$ in $\Sym(\Omega)$ under this embedding. Since $(\wt{F}',\Omega)$ is permutation isomorphic to $(F,\Omega)$, we have that $\wt{F}'^{(2),\Omega}=\wt{F}'$.

Now, fix $z\in F_1'$. Then by definition of $F_1'$ we have that for all $(\delta_1,\delta_2)\in\Delta\times\Delta$ there exists $f\in \wt{F}'$ such that $(\delta_1,\delta_2)^z=(\delta_1,\delta_2)^f$. Hence, for all $((\delta_1,g_1),(\delta_2,g_2))\in \Omega\times\Omega$, we have
$$((\delta_1,g_1),(\delta_1,g_2))^z=((\delta_1^{\mu_{g_1}(z)},g_1),(\delta_2^{\mu_{g_2}(z)},g_2))=((\delta_1^{\mu_{g_1}(f)},g_1),(\delta_2^{\mu_{g_2}(f)},g_2))=((\delta_1,g_1),(\delta_2,g_2))^f.$$
Hence, $z\in \wt{F}'^{(2),\Omega}=\wt{F}'$. Whence, $F_1'\le \wt{F}'$. Since $|F_1|=|F_1'|\le |F|\le |F_1|$, so we have $|F|=|F_1|$ and hence $F=F^{\Delta}$, as needed.
\end{proof} 

The proof of Proposition \ref{FTheorem} can be adapted to prove that the centraliser in $G$ of any normal subgroup of a totally $2$-closed group is totally $2$-closed.
\begin{proposition}\label{CTheorem} 
Let $G$ be a finite totally $2$-closed group, and let $N$ be a normal subgroup of $G$. Then $C_G(N)$ is totally $2$-closed.
\end{proposition}
\begin{proof}
 Let $C:=C_G(N)$ and let $\Delta$ be a set on which $C$ acts faithfully. Let $\Gamma:=G/C$ and $\Omega:=\Delta\times\Gamma$. Then $G$ acts faithfully on $\Omega$ by the Universal Embedding Theorem. Also, $C^{(2),\Omega}=C$ by Lemma \ref{Lemma2.1arezoomand}. The result now follows as in the second paragraph of the proof of Proposition \ref{FTheorem} above.
 \end{proof} 

\begin{proof}[Proof of Theorem B] Let $G$ be a totally $2$-closed. Then $F(G)$ is a totally $2$-closed group by Proposition \ref{FTheorem}.
Since, by \cite[1.28 Corollary]{Isaacs}, $F(G)$ is nilpotent, Theorem \ref{Theorem2arezoomand} implies that $F(G)$ is cylic or a direct product of a generalized quaternion group
with a cyclic group of odd order, as desired.\end{proof}

\section{The proof of Theorem C}
The purpose of this section is to prove Theorem C.

\begin{proof}[Proof of Theorem C]
Suppose that $G$ is a finite group of minimal order with the property that $G$ is insoluble, totally $2$-closed group, and $F:=F(G)=1$.
Write $F^*:=F^*(G)=F\circ E(G)$ for the generalized Fitting subgroup of $G$, where $E(G)$ denotes the layer of $G$.  

If $F\neq Z(G)$ then, by Proposition \ref{CTheorem}, 
$C_G(F)$ is a totally $2$-closed group  of order strictly smaller than $|G|$. Further, $C_G(F)$ contains $Z(F)$, which is non-trivial since $F$ is non-trivial. Thus, $C_G(F)$ is a finite totally $2$-closed group with non-trivial Fitting subgroup, and $|C_G(F)|<|G|$. Hence, $C_G(F)$ is soluble by hypothesis. It follows from Theorem A that $C_G(F)$ is nilpotent, so $C_G(F)\leq F$. 
Hence $F=F^*$  by \cite[9.9 Corollary]{Isaacs}. 
It follows from Theorem B that $F^*$ is either cyclic or a direct product of a
cyclic group of odd order with a generalized quaternion group. Thus, $G/Z(F^*)\leq\Aut(F^*)$ is soluble. This contradicts our choice of $G$.

So we must have that $F=Z(G)$. It follows that $F^*(G)$ has shape $Z(G)\circ T_1\circ\hdots\circ T_s$, where $T_i \unlhd G$ is a central product of (say) $t_i$ copies of a quasisimple group, and these $t_i$ copies are permuted transitively by $G$. If $i>1$, then $C_G(T_i)$ is a finite insoluble totally $2$-closed group with non-trivial Fitting subgroup by Proposition
\ref{CTheorem}. However, since $|C_G(T_i)|<|G|$, this contradicts our choice of $G$. Thus, we must have $i=1$, whence $F^*(G)$ has shape $Z(G) \circ T$, where $T\unlhd G$ is a central product of (say) $t$ copies of a finite quasimple group $S$ permuted transitively by $G$, and $Z(T)\le Z(G)$. In particular, by Theorem \ref{Theorem1arezoomand}, $Z(T)$ is cyclic.

We claim that $T/Z(T)=TZ(G)/Z(G)\cong (S/Z(S))^t$ is the unique minimal normal subgroup of $G/Z(G)$. Indeed, if $M/Z(G)$ is a minimal normal subgroup of $G/Z(G)$ with $M\neq TZ(G)$, then $M/Z(G)$ must be nonabelian (otherwise, $M/Z(G)$, and hence $M$, is nilpotent, so $M\le F(G)=Z(G)$). Write $M/Z(G)=S_1/Z(G)\times S_2/Z(G)\times\hdots\times S_e/Z(G)$, where the groups $S_i/Z(G)$ are nonabelian simple. Then choose $R_i\le S_i$ minimal with the property that $R_iZ(G)=S_i$. Then $Z(G)\cap R_i\le \Phi(R_i)$, and $R_i/R_i\cap Z(G)$ is simple. Thus, $Z(R_i)=R_i\cap Z(G)$, so $R_i/Z(R_i)$ is a nonabelian simple group, and it follows that $R_i$ is quasisimple. Since $R_i$ is subnormal in $G$, it follows that the group $R=\langle R_1,\hdots,R_e\rangle$ 
is contained in the layer $E(G)$ of $G$. Since $E(G)=T$, we then have $R\leq T$, so $M/Z(G)=RZ(G)/Z(G)=TZ(G)/Z(G)$. Thus, $M=TZ(G)$, so $TZ(G)/Z(G)$ is the unique minimal normal subgroup of $G/Z(G)$, as claimed. This completes the proof.\end{proof}


\end{document}